\documentclass[preprint,3p]{elsarticle}

\usepackage{color}
\usepackage{amsthm}
\usepackage{amssymb}
\usepackage{amsmath}
\usepackage{multirow}
\usepackage{graphicx}
\usepackage{lineno,hyperref}

\theoremstyle{definition}

\theoremstyle{plain}
\newtheorem{theorem}{Theorem}[section]
\newtheorem{lemma}{Lemma}[section]

\journal{Journal}

\begin{document}

\begin{frontmatter}

\title{A high-order L2 type difference scheme for the time-fractional diffusion equation\tnoteref{grant}}
\tnotetext[grant]{The reported study was jointly funded by RFBR (No. 20-51-53007) and NSFC (No. 12011530058)}

\author[label1]{Anatoly A. 
Alikhanov\corref{cor1}}
\cortext[cor1]{Corresponding author}
\ead{aaalikhanov@gmail.com}
\author[label2]{Chengming Huang}
\ead{chengming\_huang@hotmail.com}
\address[label1]{North-Caucasus Federal University, Pushkin str. 1,  Stavropol, 355017,  Russia}
\address[label2]{School of Mathematics and Statistics, Huazhong University of Science and Technology, Wuhan 430074, China}

	\begin{abstract}{ The present paper is devoted to constructing L2 type difference analog of the
			Caputo fractional derivative.
			The fundamental features of this difference operator are studied
			and it is used to construct difference schemes generating approximations
			of the second and fourth order in space and the $(3-\alpha)$ th-order in time
			for the time fractional diffusion equation with variable
			coefficients.  Stability of the schemes under consideration
			as well as their convergence with the rate equal
			to the order of the approximation error are proven. The received
			results are supported by the numerical computations performed for
			some test problems.}
	\end{abstract}
	
	\begin{keyword} fractional diffusion equation, finite difference
		method, stability, convergence
	\end{keyword}
	
\end{frontmatter}

\section{Introduction}
A significant growth of the researches' attention to the
fractional differential equations has been noticed lately . It is brought about by
many effective applications of fractional calculation to various
branches of science and engineering
\cite{Nakh:03,Old,Podlub:99,Hilfer:00,Kilbas:06,Uchaikin:08}. For
instance, we cannot dispense with mathematical language of fractional derivatives
when it comes to the description of the physical process of
statistical transfer which, as it is well known, brings us to diffusion
equations of fractional orders \cite{Nigma,Chuk}.

Let us consider the time fractional diffusion equation with variable
coefficients
\begin{equation}\label{ur01}
	\partial_{0t}^{\alpha}u(x,t)=\mathcal{L}u(x,t)+f(x,t) ,\,\, 0<x<l,\,\,
	0<t\leq T,
\end{equation}
\begin{equation}
	u(0,t)=0,\quad u(l,t)=0,\quad 0\leq t\leq T, \quad
	u(x,0)=u_0(x),\quad 0\leq x\leq l,\label{ur02}
\end{equation}
where
\begin{equation}
	\partial_{0t}^{\alpha}u(x,t)=\frac{1}{\Gamma(1-\alpha)}\int\limits_{0}^{t}\frac{\partial
		u(x,\eta)}{\partial\eta}(t-\eta)^{-\alpha}d\eta, \quad 0<\alpha<1
	\label{ur02.01}
\end{equation}
is the Caputo derivative of order $\alpha$,
$$
\mathcal{L}u(x,t)=\frac{\partial }{\partial
	x}\left(k(x,t)\frac{\partial u}{\partial x}\right)-q(x,t)u,
$$
$k(x,t)\geq c_1>0$, $q(x,t)\geq 0$ and $f(x,t)$ are given functions.

The time fractional diffusion equation constitutes a linear integro -
differential equation. Its solution in many cases cannot be found
in an analytical form; as a consequence it is required to apply numerical methods.
Nevertheless, in contrast to the classical case, when we numerically approximate a time
fractional diffusion equation on a certain time layer, we need information about all
the previous time layers. That is why algorithms for solving the 
time fractional diffusion equations are rather labour-consuming even in one - dimensional case.
When we pass to two - dimensional and three - dimensional
problems, their complexity grows significantly. In this respect
constructing stable differential schemes of higher order
approximation is a major task.

A common difference approximation of fractional
derivative  (\ref{ur02.01}) is the so-called $L1$ method
\cite{Old, Sun2} which is specified in the following way
\begin{equation}
    \partial_{0t_{j+1}}^{\alpha}u(x,t)=\frac{1}{\Gamma(1-\alpha)}\sum\limits_{s=0}^{j}\frac{u(x,t_{s+1})-u(x,t_{s})}{t_{s+1}-t_{s}}\int\limits_{t_{s}}^{t_{s+1}}\frac{d\eta}{(t_{j+1}-\eta)^{\alpha}}+r^{j+1}, 
    \label{ur02.02}
\end{equation}
where $0 = t_0 < t_1 < \ldots < t_{j+1}$, and $r^{j+1}$ is the local truncation
error. In the case of the uniform grid, $\tau=t_{s+1}-t_s$, for all
$s=0,1,\ldots, j+1$, it was proved that
$r^{j+1}=\mathcal{O}(\tau^{2-\alpha})$ \cite{Sun2,Lin,Alikh_arxiv3}.
The $L1$ method has been commonly used to solve the fractional
differential equations with the Caputo derivatives
\cite{Sun2,Lin,Alikh_arxiv3,ShkhTau:06,Liu:10,Alikh:12}.

The main idea of the traditional $L1$ formula for approximating Caputo fractional derivative $\partial_{0t}^\alpha f(t)$ of the function $f(t)$ is to replace the integrand $f(t)$ inside the integral by its piecewise linear interpolating polynomial (see \cite{Old, Sun2} ). A simple technique for improving the accuracy of $L1$ formula is to use piecewise high-degree interpolating polynomials instead of the linear interpolating polynomial. In general, the obtained numerical formulae in this way improve the accuracy of $L1$ formula from the order $2-\alpha$ to the order $r + 1-\alpha$, where $r \geq 2$ is the degree of the interpolating polynomial. When such formulae are applied to solve time-fractional PDEs, a key issue is the stability analysis of the corresponding methods for all $\alpha\in(0,1)$.

In \cite{Sun6} a new difference analog of the Caputo fractional
derivative with the order of approximation
$\mathcal{O}(\tau^{3-\alpha})$, called $L1-2$ formula, is
created. Based on this formula, calculations of difference
schemes for the time-fractional sub-diffusion equations in bounded
and unbounded spatial domains and the fractional ODEs are performed.
In \cite{Wang2019} the Caputo time-fractional derivative is discretized by a $(3 -\alpha)$ th-order numerical formula (called the $L2$ formula in this paper) which is constructed using piecewise quadratic interpolating polynomials. By developing a technique of discrete energy analysis, a full theoretical analysis of the stability and convergence of the method is carried out for all $\alpha\in(0,1)$.

Using piecewise quadratic interpolating polynomials, In \cite{AlikhanovJCP} a numerical formula (called $L2-1_\sigma$ formula) to approximate the Caputo fractional derivative $\partial_{0t}^\alpha f(t)$ at a special points with the numerical accuracy of order $3-\alpha$ was derived. Then some finite difference methods based on the $L2-1_\sigma$ formula were proposed for solving the time-fractional diffusion equation. In \cite{Gao, Ruilian}  $L2-1_\sigma$ formula was generalized and applied for solving the multi-term, distributed and variable order time-fractional diffusion equations. 

Difference schemes of the heightened order of approximation such as
the compact difference scheme \cite{Liu:10,Sun3,Sun4,Sun5,Wang2019} and
spectral method \cite{Lin,Lin2,Xu} were used to enhance the
spatial accuracy of fractional diffusion equations. 

By means of the energy inequality method, a priori estimates for the
solution of the Dirichlet, Robin and non-local boundary value problems for the
diffusion-wave equation with the Caputo fractional derivative have been
found in \cite{Alikh:12,Alikh:10, Alikh_nonloc}.

In the present paper we construct $L2$ type difference analog of the fractional Caputo
derivative with the order of approximation
$\mathcal{O}(\tau^{3-\alpha})$ for each  $\alpha\in(0,1)$. 
Features of the found difference operator are
investigated. Difference schemes of the second and fourth order of
approximation in space and the $(3-\alpha)$ th-order in time for the time
fractional diffusion equation with variable coefficients are
built. By means of the method of energy inequalities, the stability
and convergence  of these schemes are
proven. Numerical computations of some test problems confirming
reliability of the obtained results are implemented. The method can be
without difficulty expanded to other time fractional partial differential
equations with other boundary conditions.

\section{ The L2 type fractional numerical differentiation formula}

In this section we study a difference analog of the Caputo fractional
derivative with the  approximation order $\mathcal
O(\tau^{3-\alpha})$ and explore its fundamental features.

We consider the uniform grid  $\bar \omega_{\tau}=\{t_j=j\tau,\,
j=0,1,\ldots,M; \,T=\tau M\}$. For the Caputo fractional derivative
of the order  $\alpha$, $0<\alpha<1$, of the function  $u(t)\in
\mathcal{C}^{3}[0,T]$  at the fixed point  $t_{j+1}$,  $j\in \{1, 2,
\ldots, M-1\}$ the following equalities are valid
$$
\partial_{0t_{j+1}}^{\alpha}u(t)=\frac{1}{\Gamma(1-\alpha)}\int\limits_{0}^{t_{j+1}}\frac{u'(\eta)d\eta}{(t_{j+1}-\eta)^{\alpha}}
$$
\begin{equation}\label{ur0.99}
=
\frac{1}{\Gamma(1-\alpha)}\int\limits_{0}^{t_{2}}\frac{u'(\eta)d\eta}{(t_{j+1}-\eta)^{\alpha}}+
\frac{1}{\Gamma(1-\alpha)}\sum\limits_{s=2}^{j}\int\limits_{t_{s}}^{t_{s+1}}\frac{u'(\eta)d\eta}{(t_{j+1}-\eta)^{\alpha}}.
\end{equation}

On each interval $[t_{s-1},t_s]$ ($1\leq s\leq j$), applying the
quadratic interpolation ${\Pi}_{2,s}u(t)$ of $u(t)$ that uses three
points $(t_{s-1},u(t_{s-1}))$, $(t_{s},u(t_{s}))$ and
$(t_{s+1},u(t_{s+1}))$, we arrive at
$$
{\Pi}_{2,s}u(t)=u(t_{s-1})\frac{(t-t_{s})(t-t_{s+1})}{2\tau^2}
$$
$$
-u(t_{s})\frac{(t-t_{s-1})(t-t_{s+1})}{\tau^2}+u(t_{s+1})\frac{(t-t_{s-1})(t-t_{s})}{2\tau^2},
$$
\begin{equation}\label{ur0.991}
\left({\Pi}_{2,s}u(t)\right)'=u_{t,s}+u_{\bar tt,s}(t-t_{s+1/2}),
\end{equation}
and
\begin{equation}\label{ur0.992}
u(t)-{\Pi}_{2,s}u(t)=\frac{u'''(\bar\xi_s)}{6}(t-t_{s-1})(t-t_{s})(t-t_{s+1}),
\end{equation}
where $t\in[t_{s-1},t_{s+1}]$, $\bar\xi_s\in(t_{s-1},t_{s+1})$,
$t_{s-1/2}=t_s-0.5\tau$,  $u_{t,s}=(u(t_{s+1})-u(t_s))/\tau$,
$u_{\bar t,s}=(u(t_{s})-u(t_{s-1}))/\tau$.

In (\ref{ur0.99}), we make use of ${\Pi}_{2,s}u(t)$ in order to approximate $u(t)$ on
the interval $[t_{s-1},t_{s}]$ ($1\leq s\leq j$). In view of the equality
\begin{equation}\label{ur0.993}
\int\limits_{t_{s-1}}^{t_s}(\eta-t_{s-1/2})(t_{j+1}-\eta)^{-\alpha}d\eta=\frac{\tau^{2-\alpha}}{1-\alpha}b_{j-s+1}^{(\alpha)},
\quad 1\leq s\leq j
\end{equation}
with
$$
b_{l}^{(\alpha)}=\frac{1}{2-\alpha}\left[(l+1)^{2-\alpha}-l^{2-\alpha}\right]-\frac{1}{2}\left[(l+1)^{1-\alpha}+l^{1-\alpha}\right],\quad
l\geq 0,
$$
from (\ref{ur0.99}) and  (\ref{ur0.991}) we get the difference
analog of the Caputo fractional derivative of order $\alpha$
($0<\alpha<1$) for the function $u(t)$, at the points $t_{j+1}$
($j=1, 2, \ldots$), in this form:
$$
\partial_{0t_{j+1}}^{\alpha}u(\eta)=
\frac{1}{\Gamma(1-\alpha)}\int\limits_{0}^{t_{2}}\frac{u'(\eta)d\eta}{(t_{j+1}-\eta)^{\alpha}}+
\frac{1}{\Gamma(1-\alpha)}\sum\limits_{s=2}^{j}\int\limits_{t_{s}}^{t_{s+1}}\frac{u'(\eta)d\eta}{(t_{j+1}-\eta)^{\alpha}}
$$
$$
\approx
\frac{1}{\Gamma(1-\alpha)}\int\limits_{0}^{t_{2}}\frac{\left({\Pi}_{2,1}u(\eta)\right)'d\eta}{(t_{j+1}-\eta)^{\alpha}}+
\frac{1}{\Gamma(1-\alpha)}\sum\limits_{s=2}^{j}\int\limits_{t_{s}}^{t_{s+1}}\frac{\left({\Pi}_{2,s}u(\eta)\right)'d\eta}{(t_{j+1}-\eta)^{\alpha}}
$$
$$
=\frac{1}{\Gamma(1-\alpha)}\int\limits_{0}^{t_{2}}\frac{u_{t,1}+u_{\bar
tt,1}(\eta-t_{3/2})}{(t_{j+1}-\eta)^{\alpha}}d\eta
$$
$$
+\frac{1}{\Gamma(1-\alpha)}\sum\limits_{s=2}^{j}\int\limits_{t_{s}}^{t_{s+1}}\frac{u_{t,s}+u_{\bar
tt,s}(\eta-t_{s+1/2})}{(t_{j+1}-\eta)^{\alpha}}d\eta
$$
$$
=\frac{\tau^{1-\alpha}}{\Gamma{(2-\alpha)}}\left((a_{j}^{(\alpha)}-b_{j}^{(\alpha)}-b_{j-1}^{(\alpha)})u_{t,0}+(a_{j-1}^{(\alpha)}+b_{j-1}^{(\alpha)}+b_{j}^{(\alpha)})u_{t,1}\right)
$$
$$
+\frac{\tau^{1-\alpha}}{\Gamma{(2-\alpha)}}\sum\limits_{s=2}^{j}(-b_{j-s}^{(\alpha)}u_{t,s-1}+(a_{j-s}^{(\alpha)}+b_{j-s}^{(\alpha)})u_{
t,s})
$$
$$
=\frac{\tau^{1-\alpha}}{\Gamma{(2-\alpha)}}\left((a_{j}^{(\alpha)}-b_{j}^{(\alpha)}-b_{j-1}^{(\alpha)})u_{t,0}+(a_{j-1}^{(\alpha)}+b_{j-1}^{(\alpha)}+b_{j}^{(\alpha)}-b_{j-2}^{(\alpha)})u_{t,1}\right)
$$
$$
+\frac{\tau^{1-\alpha}}{\Gamma{(2-\alpha)}}\left(\sum\limits_{s=2}^{j-1}(-b_{j-s-1}^{(\alpha)}+a_{j-s}^{(\alpha)}+b_{j-s}^{(\alpha)})u_{
t,s}+(a_{0}^{(\alpha)}+b_{0}^{(\alpha)})u_{ t,j}\right)=
$$
\begin{equation}
=\frac{\tau^{1-\alpha}}{\Gamma{(2-\alpha)}}\sum\limits_{s=0}^{j}c_{j-s}^{(\alpha)}u_{t,s}=\Delta_{0t_{j+1}}^\alpha
u, \label{ur0.994}
\end{equation}
where
$$
a_{l}^{(\alpha)}=(l+1)^{1-\alpha}-l^{1-\alpha}, \quad l\geq 0;
$$

for $j=1$
\begin{equation}
c_{s}^{(\alpha)}=
\begin{cases}
a_{0}^{(\alpha)}+b_{0}^{(\alpha)}+b_{1}^{(\alpha)}, \quad\quad s=0,\\
a_{1}^{(\alpha)}-b_{1}^{(\alpha)}-b_{0}^{(\alpha)}, \quad\quad s=1,
\label{ur0.995.1}
\end{cases}
\end{equation}

for $j=2$
\begin{equation}
c_{s}^{(\alpha)}=
\begin{cases}
a_{0}^{(\alpha)}+b_{0}^{(\alpha)}, \quad\quad\quad\quad\quad\quad\,\,\,\, s=0,\\
a_{1}^{(\alpha)}+b_{1}^{(\alpha)}+b_{2}^{(\alpha)}-b_{0}^{(\alpha)},
\quad\, s=1,\\
a_{2}^{(\alpha)}-b_{2}^{(\alpha)}-b_{1}^{(\alpha)},
\quad\quad\quad\quad s=2,\label{ur0.995.2}
\end{cases}
\end{equation}

and for $j\geq 3$,
\begin{equation}
c_{s}^{(\alpha)}=
\begin{cases}
a_{0}^{(\alpha)}+b_{0}^{(\alpha)}, \quad\quad\quad\quad\quad\quad\quad \, s=0,\\
a_{s}^{(\alpha)}+b_{s}^{(\alpha)}-b_{s-1}^{(\alpha)}, \quad\quad\quad \, \quad 1\leq s\leq j-2,\\
a_{j-1}^{(\alpha)}+b_{j-1}^{(\alpha)}+b_{j}^{(\alpha)}-b_{j-2}^{(\alpha)}, \quad s=j-1,\\
a_{j}^{(\alpha)}-b_{j}^{(\alpha)}-b_{j-1}^{(\alpha)},
\quad\quad\quad\quad\, s=j. \label{ur0.995.3}
\end{cases}
\end{equation}

We name the fractional numerical differentiation formula
(\ref{ur0.994}) for the Caputo fractional derivative of order
$\alpha$ ($0<\alpha<1$) the L2 formula.

\begin{lemma} \label{lem1} For any $\alpha\in(0,1)$, $j=1, 2, \ldots, M-1$
and $u(t)\in \mathcal{C}^3[0,t_{j+1}]$
\begin{equation}
|\partial_{0t_{j+1}}^{\alpha}u-\Delta_{0t_{j+1}}^\alpha
u|=\mathcal{O}(\tau^{3-\alpha}).
 \label{ur0.996}
\end{equation}
\end{lemma}
\begin{proof} Let
$\partial_{0t_{j+1}}^{\alpha}u-\Delta_{0t_{j+1}}^\alpha
u=R_{0}^{2}+R_{2}^{j+1}$, where
$$
R_{0}^{2}=\frac{1}{\Gamma(1-\alpha)}\int\limits_{0}^{t_{2}}\frac{u'(\eta)d\eta}{(t_{j+1}-\eta)^{\alpha}}-
\frac{1}{\Gamma(1-\alpha)}\int\limits_{0}^{t_{2}}\frac{\left({\Pi}_{2,1}u(\eta)\right)'d\eta}{(t_{j+1}-\eta)^{\alpha}}
$$
$$
=\frac{1}{\Gamma(1-\alpha)}\int\limits_{0}^{t_{2}}\frac{\left(u(\eta)-{\Pi}_{2,1}u(\eta)\right)'d\eta}{(t_{j+1}-\eta)^{\alpha}}=
-\frac{\alpha}{\Gamma(1-\alpha)}\int\limits_{0}^{t_{2}}\frac{\left(u(\eta)-{\Pi}_{2,1}u(\eta)\right)d\eta}{(t_{j+1}-\eta)^{\alpha+1}}
$$
$$
=-\frac{\alpha}{6\Gamma(1-\alpha)}\int\limits_{0}^{t_{2}}{u'''(\bar\xi_1)\eta(\eta-t_1)(\eta-t_2)(t_{j+1}-\eta)^{-\alpha-1}}d\eta,
$$
$$
R_{2}^{j+1}=\frac{1}{\Gamma(1-\alpha)}\sum\limits_{s=2}^{j}\int\limits_{t_{s}}^{t_{s+1}}\frac{u'(\eta)d\eta}{(t_{j+1}-\eta)^{\alpha}}-
\frac{1}{\Gamma(1-\alpha)}\sum\limits_{s=2}^{j}\int\limits_{t_{s}}^{t_{s+1}}\frac{\left({\Pi}_{2,s}u(\eta)\right)'d\eta}{(t_{j+1}-\eta)^{\alpha}}
$$
$$
=\frac{1}{\Gamma(1-\alpha)}\sum\limits_{s=2}^{j}\int\limits_{t_{s}}^{t_{s+1}}\left(u(\eta)-{\Pi}_{2,s}u(\eta)\right)'{(t_{j+1}-\eta)^{-\alpha}}d\eta
$$
$$
=-\frac{\alpha}{\Gamma(1-\alpha)}\sum\limits_{s=2}^{j}\int\limits_{t_{s}}^{t_{s+1}}\left(u(\eta)-{\Pi}_{2,s}u(\eta)\right){(t_{j+1}-\eta)^{-\alpha-1}}d\eta
$$
$$
=-\frac{\alpha}{6\Gamma(1-\alpha)}\sum\limits_{s=2}^{j}\int\limits_{t_{s}}^{t_{s+1}}u'''(\bar\xi_s)(\eta-t_{s-1})(\eta-t_{s})(\eta-t_{s+1}){(t_{j+1}-\eta)^{-\alpha-1}}d\eta,
$$

Next we estimate the errors $R_{0}^{2}$ and $R_{2}^{j+1}$:

For $j=1$ we have
$$
\left|R_{0}^{2}\right|
=\frac{\alpha}{6\Gamma(1-\alpha)}\left|\int\limits_{0}^{t_{2}}{u'''(\bar\xi_1)\eta(\eta-t_1)(\eta-t_2)(t_{2}-\eta)^{-\alpha-1}}d\eta\right|
$$
$$
\leq\frac{\alpha M_3
\tau^2}{3\Gamma(1-\alpha)}\int\limits_{0}^{t_{2}}(t_{2}-\eta)^{-\alpha}d\eta=\frac{2^{1-\alpha}\alpha
M_3}{3\Gamma(2-\alpha)}\tau^{3-\alpha},
$$

For $j\geq2$ we have
$$
\left|R_{0}^{2}\right|
=\frac{\alpha}{6\Gamma(1-\alpha)}\left|\int\limits_{0}^{t_{2}}{u'''(\bar\xi_1)\eta(\eta-t_1)(\eta-t_2)(t_{j+1}-\eta)^{-\alpha-1}}d\eta\right|
$$
$$
\leq\frac{2\sqrt{3}\alpha M_3
\tau^3}{54\Gamma(1-\alpha)}\int\limits_{0}^{t_{2}}{(t_{j+1}-\eta)^{-\alpha-1}}d\eta=\frac{\sqrt{3}
M_3
\tau^{3-\alpha}}{27\Gamma(1-\alpha)}\left((j-1)^{-\alpha}-(j+1)^{-\alpha}\right)
\leq\frac{\sqrt{3}(1-3^{-\alpha})M_3}{27\Gamma(1-\alpha)}\tau^{3-\alpha},
$$

$$
\left|R_{2}^{j+1}\right|=\frac{\alpha}{6\Gamma(1-\alpha)}\left|\sum\limits_{s=2}^{j}\int\limits_{t_{s}}^{t_{s+1}}u'''(\bar\xi_s)(\eta-t_{s-1})(\eta-t_{s})(\eta-t_{s+1}){(t_{j+1}-\eta)^{-\alpha-1}}d\eta\right|
$$
$$
\leq\frac{2\sqrt{3}\alpha M_3
\tau^3}{54\Gamma(1-\alpha)}\sum\limits_{s=2}^{j-1}\int\limits_{t_{s}}^{t_{s+1}}{(t_{j+1}-\eta)^{-\alpha-1}}d\eta+
\frac{\alpha M_3
\tau^2}{3\Gamma(1-\alpha)}\int\limits_{t_{j}}^{t_{j+1}}{(t_{j+1}-\eta)^{-\alpha}}d\eta
$$
$$
=\frac{\sqrt{3}\alpha M_3
\tau^3}{27\Gamma(1-\alpha)}\int\limits_{t_{2}}^{t_{j}}{(t_{j+1}-\eta)^{-\alpha-1}}d\eta+
\frac{\alpha M_3
\tau^2}{3\Gamma(1-\alpha)}\int\limits_{t_{j}}^{t_{j+1}}{(t_{j+1}-\eta)^{-\alpha}}d\eta
$$
$$
=\frac{\sqrt{3} M_3
\tau^3}{27\Gamma(1-\alpha)}\left(\tau^{-\alpha}-t_{j-1}^{-\alpha}\right)+
\frac{\alpha M_3
\tau^2}{3\Gamma(1-\alpha)}\frac{\tau^{1-\alpha}}{1-\alpha}
\leq
\left(\frac{\sqrt{3}}{9}+\frac{\alpha}{(1-\alpha)}\right)\frac{M_3}{3\Gamma(1-\alpha)}\tau^{3-\alpha}.
$$
\end{proof}

\subsection{Fundamental features of the new L2 fractional numerical differentiation formula.}

\begin{lemma}\label{lm_pr_1} For all $\alpha\in(0,1)$ and $s=1, 2, 3, \ldots$
\begin{equation}
\frac{1-\alpha}{(s+1)^\alpha}<a_s<\frac{1-\alpha}{s^\alpha},
 \label{url2_1}
\end{equation}
\begin{equation}
\frac{\alpha(1-\alpha)}{(s+2)^{\alpha+1}}<a_s-a_{s+1}<\frac{\alpha(1-\alpha)}{s^{\alpha+1}},
 \label{url2_2}
\end{equation}
\begin{equation}
\frac{\alpha(1-\alpha)}{12(s+1)^{\alpha+1}}<b_{s}<\frac{\alpha(1-\alpha)}{12s^{\alpha+1}},
 \label{url2_3}
\end{equation}
\end{lemma}

\begin{proof} The validity of Lemma \ref{lm_pr_1} follows from the following
equalities:
$$
a_{s}=(1-\alpha)\int\limits_{0}^{1}\frac{d\xi}{(s+\xi)^{\alpha}},
$$
$$
a_{s}-a_{s+1}=\alpha(1-\alpha)\int\limits_{0}^{1}d\eta\int\limits_{0}^{1}\frac{d\xi}{(s+\xi+\eta)^{\alpha+1}},
$$
$$
b_{s}=\frac{\alpha(1-\alpha)}{2^{2-\alpha}}\int\limits_{0}^{1}\eta
d\eta\int\limits_{2s+1-\eta}^{2s+1+\eta}\frac{d\xi}{\xi^{\alpha+1}}.
$$
\end{proof}

For $j=1$ we have
$$
c_0^{(\alpha)}=\frac{2+\alpha}{2^{\alpha}(2-\alpha)}, \quad
c_1^{(\alpha)}=\frac{2-3\alpha}{2^{\alpha}(2-\alpha)}, \quad
c_0^{(\alpha)}+3c_1^{(\alpha)}=\frac{2^{3-\alpha}(1-\alpha)}{2-\alpha}>0.
$$

For $j\geq2$, the next lemma shows properties of the coefficient
$c_s^{(\alpha)}$ defined in (\ref{ur0.995.2}) and (\ref{ur0.995.3})

\begin{lemma} For any $\alpha\in(0,1)$ and $c_s^{(\alpha)}$
($0\leq s\leq j$, $j\geq 2$) the following inequalities are valid
\begin{equation}
\frac{11}{16}\cdot\frac{1-\alpha}{(j+1)^\alpha} < c_j^{(\alpha)} < \frac{1-\alpha}{j^\alpha},
\label{url2_4}
\end{equation}
\begin{equation}
c_0^{(\alpha)}>c_2^{(\alpha)}>c_3^{(\alpha)}>\ldots>c_{j-2}^{(\alpha)}>c_{j-1}^{(\alpha)}>c_j^{(\alpha)},
\label{url2_5}
\end{equation}
\begin{equation}
c_0^{(\alpha)}+3c_1^{(\alpha)}-4c_2^{(\alpha)}>0. \label{url2_6}
\end{equation}
\end{lemma}

\begin{proof}  For $j\geq2$ we get
$$
c_j^{(\alpha)}=a_j^{(\alpha)}-b_j^{(\alpha)}-b_{j-1}^{(\alpha)}<a_j^{(\alpha)}<
\frac{1-\alpha}{j^\alpha}.
$$
$$
c_j^{(\alpha)}=a_j^{(\alpha)}-b_j^{(\alpha)}-b_{j-1}^{(\alpha)}>
\frac{1-\alpha}{(j+1)^\alpha}-\frac{\alpha(1-\alpha)}{12j^{\alpha+1}}-\frac{\alpha(1-\alpha)}{12(j-1)^{\alpha+1}}.
$$
$$
=\frac{1-\alpha}{(j+1)^\alpha}\left(1-\frac{\alpha}{12}\cdot\left(\frac{j+1}{j}\right)^\alpha\cdot\frac{1}{j}-\frac{\alpha}{12}\cdot\left(\frac{j+1}{j-1}\right)^\alpha\cdot\frac{1}{j-1}\right)
$$
$$
>\frac{1-\alpha}{(j+1)^\alpha}\left(1-\frac{1}{12}\cdot\frac{3}{2}\cdot\frac{1}{2}-\frac{1}{12}\cdot3\right)=\frac{11}{16}\cdot\frac{1-\alpha}{(j+1)^\alpha}
$$
Inequality (\ref{url2_4}) is proved. Let us prove inequality
(\ref{url2_5}).
$$
c_0^{(\alpha)}-c_2^{(\alpha)}\geq
a_0^{(\alpha)}-a_2^{(\alpha)}+b_0^{(\alpha)}-
b_2^{(\alpha)}+b_1^{(\alpha)}-b_3^{(\alpha)}>0,
$$

For $j \geq 5$, $2\leq s\leq j-3$ we have
$$
c_{s}^{(\alpha)}-c_{s+1}^{(\alpha)}=a_{s}^{(\alpha)}-a_{s+1}^{(\alpha)}-b_{s-1}^{(\alpha)}+2b_{s}^{(\alpha)}-b_{s+1}^{(\alpha)}
$$
$$
=\frac{1}{2-\alpha}\left(-(s+2)^{2-\alpha}+3(s+1)^{2-\alpha}-3s^{2-\alpha}+(s-1)^{2-\alpha}\right)
$$
$$
-\frac{1}{2}\left(-(s+2)^{1-\alpha}+3(s+1)^{1-\alpha}-3s^{1-\alpha}+(s-1)^{1-\alpha}\right)
$$
$$
=\alpha(1-\alpha)\int\limits_{0}^{1}dz_1\int\limits_{0}^{1}dz_2\int\limits_{0}^{1}\frac{dz_3}{(s-1+z_1+z_2+z_3)^{\alpha+1}}
$$
$$
-\frac{\alpha(1-\alpha)(1+\alpha)}{2}\int\limits_{0}^{1}dz_1\int\limits_{0}^{1}dz_2\int\limits_{0}^{1}\frac{dz_3}{(s-1+z_1+z_2+z_3)^{\alpha+2}}
$$
$$
=\alpha(1-\alpha)\int\limits_{0}^{1}dz_1\int\limits_{0}^{1}dz_2\int\limits_{0}^{1}
\frac{\left(1-\frac{1+\alpha}{2}\cdot\frac{1}{s-1+z_1+z_2+z_3}\right)}{(s-1+z_1+z_2+z_3)^{\alpha+1}}dz_3
$$
$$
>\frac{\alpha(1-\alpha)}{(s+2)^{\alpha+1}}\left(1-\frac{1+\alpha}{2}\int\limits_{0}^{1}dz_1\int\limits_{0}^{1}dz_2\int\limits_{0}^{1}
\frac{dz_3}{1+z_1+z_2+z_3}\right).
$$
Since
$$
\int\limits_{0}^{1}dz_1\int\limits_{0}^{1}dz_2\int\limits_{0}^{1}
\frac{dz_3}{1+z_1+z_2+z_3}=\frac{1}{2}\left(44\ln{2}-27\ln{3}\right)<\frac{1}{2},
$$
$$
c_{s}^{(\alpha)}-c_{s+1}^{(\alpha)}>\frac{\alpha(1-\alpha)}{(s+2)^{\alpha+1}}\left(1-\frac{1+\alpha}{4}\right)>\frac{\alpha(1-\alpha)}{2(s+2)^{\alpha+1}}>0.
$$
For $j \geq 4$ we get
$$
c_{j-2}^{(\alpha)}-c_{j-1}^{(\alpha)}=a_{j-2}^{(\alpha)}-a_{j-1}^{(\alpha)}-b_{j-3}^{(\alpha)}+2b_{j-2}^{(\alpha)}-b_{j-1}^{(\alpha)}-b_{j}^{(\alpha)}
$$
$$
>\frac{\alpha(1-\alpha)}{2j^{\alpha+1}}-b_{j}^{(\alpha)}>\frac{\alpha(1-\alpha)}{2j^{\alpha+1}}-\frac{\alpha(1-\alpha)}{12j^{\alpha+1}}=\frac{5\alpha(1-\alpha)}{12j^{\alpha+1}}>0.
$$
For $j \geq 3$ we have
$$
c_{j-1}^{(\alpha)}-c_{j}^{(\alpha)}=a_{j-1}^{(\alpha)}-a_{j}^{(\alpha)}-b_{j-2}^{(\alpha)}+2b_{j-1}^{(\alpha)}+2b_{j}^{(\alpha)}
$$
$$
>a_{j-1}^{(\alpha)}-a_{j}^{(\alpha)}-b_{j-2}^{(\alpha)}+2b_{j-1}^{(\alpha)}-b_{j}^{(\alpha)}>\frac{\alpha(1-\alpha)}{2(j+1)^{\alpha+1}}>0.
$$
Inequality (\ref{url2_5}) is proved.

For $j=2$ we get
$$
c_{0}^{(\alpha)}+3c_{1}^{(\alpha)}-4c_{2}^{(\alpha)}=a_{0}^{(\alpha)}+3a_{1}^{(\alpha)}-4a_{2}^{(\alpha)}-2b_{0}^{(\alpha)}+7b_{1}^{(\alpha)}+7b_{2}^{(\alpha)}
$$
$$
>
a_{0}^{(\alpha)}-a_{1}^{(\alpha)}-2b_{0}^{(\alpha)}=3-2^{1-\alpha}-\frac{2}{2-\alpha}.
$$
Since, for any function $f(x)\in C^2[0,1]$, if $f(0)=0$, $f(1)=0$
and $f''(x)<0$ for all $x\in(0,1)$ then $f(x)>0$ for all
$x\in(0,1)$, we have
$$
c_{0}^{(\alpha)}+3c_{1}^{(\alpha)}-4c_{2}^{(\alpha)}>f(\alpha) = 3-2^{1-\alpha}-\frac{2}{2-\alpha}>0
\quad \text{for all}\quad \alpha\in(0,1).
$$
For $j=3$ we get
$$
c_{0}^{(\alpha)}+3c_{1}^{(\alpha)}-4c_{2}^{(\alpha)}=a_{0}^{(\alpha)}+3a_{1}^{(\alpha)}-4a_{2}^{(\alpha)}-2b_{0}^{(\alpha)}+7b_{1}^{(\alpha)}-4b_{2}^{(\alpha)}-4b_{3}^{(\alpha)}
$$
$$
>a_{0}^{(\alpha)}-a_{1}^{(\alpha)}-2b_{0}^{(\alpha)}+4(a_{1}^{(\alpha)}-a_{2}^{(\alpha)}-b_{3}^{(\alpha)})>3-2^{1-\alpha}-\frac{2}{2-\alpha}>0.
$$
\end{proof}

\begin{lemma} \label{lm_ineq} For any real constants $c_0, c_1$ such that
$c_0\geq\max\{c_1,-3c_1\}$, and $\{v_j\}_{j=0}^{j=M}$ the following
inequality holds
\begin{equation}
v_{j+1}\left(c_0v_{j+1}-(c_0-c_1)v_{j}-c_1v_{j-1}\right)\geq
E_{j+1}-E_{j}, \quad j=1, \ldots, M-1, \label{url5}
\end{equation}
where
$$
E_{j}=\left(\frac{1}{2}\sqrt{\frac{c_0-c_1}{2}}+\frac{1}{2}\sqrt{\frac{c_0+3c_1}{2}}\right)^2v_{j}^{2}
+\left(\sqrt{\frac{c_0-c_1}{2}}v_{j}-\left(\frac{1}{2}\sqrt{\frac{c_0-c_1}{2}}+\frac{1}{2}\sqrt{\frac{c_0+3c_1}{2}}\right)v_{j-1}\right)^2,
\quad j=1, 2, \ldots, M.
$$
\end{lemma} 

\begin{proof} The proof of Lemma \ref{lm_ineq}  immediately follows from the next
equality
$$
v_{j+1}\left(c_0v_{j+1}-(c_0-c_1)v_{j}-c_1v_{j-1}\right)-
E_{j+1}+E_{j}
$$
$$
=\left(\left(\frac{1}{2}\sqrt{\frac{c_0-c_1}{2}}-\frac{1}{2}\sqrt{\frac{c_0+3c_1}{2}}\right)v_{j+1}-\sqrt{\frac{c_0-c_1}{2}}v_{j}+\left(\frac{1}{2}\sqrt{\frac{c_0-c_1}{2}}+\frac{1}{2}\sqrt{\frac{c_0+3c_1}{2}}\right)v_{j-1}\right)^2\geq0.
$$
\end{proof}

\begin{lemma} \cite{AlikhanovJCP} If
	$g_{j}^{j+1}\geq g_{j-1}^{j+1}\geq \ldots\geq g_{0}^{j+1}>0$, $j=0,1,\ldots,M-1$
	then for any function $v(t)$ defined on the grid $\overline
	\omega_{\tau}$ one has the inequalities
	\begin{equation}\label{ur04}
		v^{j+1}{_g}\Delta_{0t_{j+1}}^{\alpha}v\geq
		\frac{1}{2}{_g}\Delta_{0t_{j+1}}^{\alpha}(v^2)+\frac{1}{2g^{j+1}_j}\left({_g}\Delta_{0t_{j+1}}^{\alpha}v\right)^2,
	\end{equation}
	where
	$$
	{_g}\Delta_{0t_{j+1}}^{\alpha}y_i=\sum\limits_{s=0}^{j}\left(y_i^{s+1}-y_i^s\right)g_{s}^{j+1}, 
	$$
	is a difference analog of the Caputo fractional derivative
	of the order $\alpha$ ($0<\alpha<1$).
\end{lemma}

\begin{lemma}\label{lem16}  For any function $v(t)$ defined on the grid
$\overline \omega_{\tau}$ one has the inequality
\begin{equation}\label{url6}
v_{j+1}\Delta_{0t_{j+1}}^{\alpha}v\geq\frac{\tau^{-\alpha}}{\Gamma(2-\alpha)}\left(E_{j+1}-E_{j}\right)+\frac{1}{2}\bar{\Delta}_{0t_{j+1}}^{\alpha}v^2 =
\frac{\tau^{-\alpha}}{\Gamma(2-\alpha)}\left(\mathcal{E}_{j+1}-\mathcal{E}_{j}\right) - \frac{\tau^{-\alpha}}{2\Gamma(2-\alpha)}\bar c^{(\alpha)}_j v_0^2,
\end{equation}
where
$$
\bar{\Delta}_{0t_{j+1}}^{\alpha}v=\frac{\tau^{-\alpha}}{\Gamma{(2-\alpha)}}\sum\limits_{s=0}^{j}\bar{c}_{j-s}^{(\alpha)}(v_{s+1}-v_{s}),\quad
j=1, 2, \ldots, M,
$$
$$
\bar{c}_{0}^{(\alpha)}={c}_{2}^{(\alpha)},\quad
\bar{c}_{1}^{(\alpha)}={c}_{2}^{(\alpha)},\quad
\bar{c}_{s}^{(\alpha)}={c}_{s}^{(\alpha)},\quad s=2,3,\ldots,j,
$$
$$
\text{for}\quad j=1, 2, 3, \ldots, M, \quad
E_j=E_j({c}_{0}^{(\alpha)}-{c}_{2}^{(\alpha)},{c}_{1}^{(\alpha)}-{c}_{2}^{(\alpha)}),
$$
$$
\mathcal{E}_{j} = E_{j} + \frac{1}{2}\sum\limits_{s=0}^{j-1}\bar{c}_{j-1-s}^{(\alpha)}v_{s+1}^2.
$$
\end{lemma}

\begin{proof} For $j=1$ we have
$$
v_{2}\Delta_{0t_{2}}^{\alpha}v=\frac{\tau^{-\alpha}}{\Gamma(2-\alpha)}v_2\left(c_{0}^{(\alpha)}(v_2-v_1)+c_{1}^{(\alpha)}(v_1-v_0)\right)
$$
$$
=\frac{\tau^{-\alpha}}{\Gamma(2-\alpha)}v_2\left((c_{0}^{(\alpha)}-c_{2}^{(\alpha)})v_2-(c_{0}^{(\alpha)}-c_{1}^{(\alpha)})v_1-(c_{1}^{(\alpha)}-c_{2}^{(\alpha)})v_0\right)
$$
$$
+
\frac{\tau^{-\alpha}}{\Gamma(2-\alpha)}c_{2}^{(\alpha)}\left(v_2^2-v_2v_0\right)
$$
$$
\geq\frac{\tau^{-\alpha}}{\Gamma(2-\alpha)}\left(E_{2}-E_{1}\right)+\frac{\tau^{-\alpha}}{2\Gamma(2-\alpha)}c_{2}^{(\alpha)}\left(v_2^2-v_0^2\right)
$$
$$
=\frac{\tau^{-\alpha}}{\Gamma(2-\alpha)}\left(E_{2}-E_{1}\right)+\frac{1}{2}\bar{\Delta}_{0t_{2}}^{\alpha}v^2.
$$
For $j=2, 3, \ldots, M-1$ we have
$$
v_{j+1}\Delta_{0t_{j+1}}^{\alpha}v=\frac{\tau^{-\alpha}}{\Gamma(2-\alpha)}v_{j+1}\sum\limits_{s=0}^{j}{c}_{j-s}^{(\alpha)}(v_{s+1}-v_{s})
$$
$$
\frac{\tau^{-\alpha}}{\Gamma(2-\alpha)}v_{j+1}\left((c_{0}^{(\alpha)}-c_{2}^{(\alpha)})(v_{j+1}-v_{j})+(c_{1}^{(\alpha)}-c_{2}^{(\alpha)})(v_{j}-v_{j-1})\right)+v_{j+1}\bar{\Delta}_{0t_{j+1}}^{\alpha}v
$$
$$
=\frac{\tau^{-\alpha}}{\Gamma(2-\alpha)}v_{j+1}\left((c_{0}^{(\alpha)}-c_{2}^{(\alpha)})v_{j+1}-(c_{0}^{(\alpha)}-c_{1}^{(\alpha)})v_{j}-(c_{1}^{(\alpha)}-c_{2}^{(\alpha)})v_{j-1}\right)+v_{j+1}\bar{\Delta}_{0t_{j+1}}^{\alpha}v
$$
$$
\geq\frac{\tau^{-\alpha}}{\Gamma(2-\alpha)}\left(E_{j+1}-E_{j}\right)+\frac{1}{2}\bar{\Delta}_{0t_{j+1}}^{\alpha}v^2.
$$
In addition, the following equality holds
$$
\bar{\Delta}_{0t_{j+1}}^{\alpha}v^2 = 
\frac{\tau^{-\alpha}}{\Gamma{(2-\alpha)}}\sum\limits_{s=0}^{j}\bar{c}_{j-s}^{(\alpha)}(v_{s+1}^2-v_{s}^2) =
\frac{\tau^{-\alpha}}{\Gamma{(2-\alpha)}}\left(
\sum\limits_{s=0}^{j}\bar{c}_{j-s}^{(\alpha)}v_{s+1}^2 - \sum\limits_{s=0}^{j-1}\bar{c}_{j-1-s}^{(\alpha)}v_{s+1}^2 -
\bar{c}_{j}^{(\alpha)}v_{0}^2\right).
$$
\end{proof}

\section{A difference scheme for the time fractional diffusion equation}

In this section for problem  (\ref{ur01})--(\ref{ur02})   a
difference scheme with the approximation order
$\mathcal{O}(h^2+\tau^{3-\alpha})$ is constructed. The stability of the
constructed difference scheme as well as its convergence in the grid
$L_2$ - norm with the rate equal to the order of the approximation
error is proved. The obtained results are supported with numerical
calculations carried out for a test example.

\subsection{Derivation  of the difference scheme}

\begin{lemma} \cite{AlikhanovJCP} For any functions  $k(x)\in \mathcal{C}_{x}^{3}$
and $v(x)\in\mathcal{C}_{x}^{4}$ the following equality holds true:
\begin{equation}\label{ur16.1}
	\left.\frac{d}{dx}\left(k(x)\frac{d}{dx}v(x)\right)\right|_{x=x_i}=\frac{k(x_{i+1/2})v(x_{i+1})-(k(x_{i+1/2})+k(x_{i-1/2}))v(x_{i})+k(x_{i-1/2})v(x_{i-1})}{h^2}+\mathcal{O}(h^2).
\end{equation}
\label{lemh2}
\end{lemma}
Let  $u(x,t)\in \mathcal{C}_{x,t}^{4,3}$ be a solution of problem
(\ref{ur01})--(\ref{ur02}).  Then we consider equation  (\ref{ur01})
for $(x,t)=(x_i,t_{j+1})\in\overline Q_T$,\,
$i=1,2,\ldots,N-1$,\, $j=1,2,\ldots,M-1$:
\begin{equation}\label{ur17}
	\partial_{0t_{j+1}}^{\alpha} u=\left.\mathcal{L}u(x,t)\right|_{(x_i,t_{j+1})}+f(x_i,t_{j+1}).
\end{equation}

On the basis of Lemmas \ref{lem1} and \ref{lemh2} we have
$$
\partial_{0t_{j+1}}^{\alpha} u= \Delta_{0t_{j+\sigma}}^{\alpha}u + \mathcal{O}(\tau^{3-\alpha})
$$
$$
\left.\mathcal{L}u(x,t)\right|_{(x_i,t_{j+1})}=\Lambda
u(x_i,t_{j+1})+\mathcal{O}(h^2),
$$
where the difference operator $\Lambda$ is defined as follows
$$
	(\Lambda y)_i=\left((ay_{\bar
	x})_x-dy\right)_i=\frac{a_{i+1}y_{i+1}-(a_{i+1}+a_i)y_i+a_iy_{i-1}}{h^2}-d_iy_i,
$$
$$
 y_{\bar x,i}=\frac{y_i-y_{i-1}}{h},\quad y_{x,i}=\frac{y_{i+1}-y_{i}}{h},
$$
with the coefficients
$a_i^{j+1}=k(x_{i-1/2},t_{j+1})$,\,
$d_i^{j+1}=q(x_{i},t_{j+1})$. Let
$\varphi_i^{j+1}=f(x_i,t_{j+1})$, then
we get the difference scheme with the approximation order
$\mathcal{O}(h^2+\tau^{3-\alpha})$:
\begin{equation}\label{ur18}
	\Delta_{0t_{j+1}}^{\alpha}y_i=\Lambda y^{j+1}_i
	+\varphi_i^{j+1}, \quad i=1,2,\ldots,N-1,\quad j=1,2,\ldots,M-1,
\end{equation}
\begin{equation}
	y(0,t)=0,\quad y(l,t)=0,\quad t\in \overline \omega_{\tau}, \quad
	y(x,0)=u_0(x),\quad  x\in \overline \omega_{h},\label{ur19}
\end{equation}

\textbf{Remark.} We assume that the solution $y^1_i$  is
found with the order of accuracy $\mathcal{O}(h^4+\tau^{3-\alpha})$. For
example, we can use $L1$-formula and solve problem (1.2)-(1.4) on the
time layer $[0,\tau]$ with step $\tau_1=\mathcal{O}(\tau^{\frac{3-\alpha}{2-\alpha}})$.

\subsection{Stability and convergence}

\begin{theorem}\label{thm_1} The difference scheme (\ref{ur18})--(\ref{ur19})
is unconditionally stable and its solution meets the following a
priori estimates:
\begin{equation}\label{ur20}
	\sum\limits_{j=1}^{M-1}\left(\|y^{j+1}\|_0^2 + \|y_{\bar x}^{j+1}]|_0^2\right)\tau \leq M_1\left(\|y^1\|_0^2 +\|y^0\|_0^2+\sum\limits_{j=1}^{M-1}\|\varphi^{j+1}\|_0^2\tau\right),
\end{equation}
where $\|y]|_0^2=\sum\limits_{i=1}^{N}y_i^2h$, $M_1 > 0$ is a known number independent of $h$ and $\tau$. 
\end{theorem}

\begin{proof}   Taking the inner product of the equation
(\ref{ur18}) with $y^{j+1}$, we have
\begin{equation}\label{ur18_1}
	\left(y^{{j+1}},\Delta_{0t_{j+1}}^\alpha y\right)-\left(y^{{j+1}},\Lambda
	y^{{j+1}}\right)=\left(y^{{j+1}},\varphi^{j+1}\right).
\end{equation}

Using Lemma \ref{lem16}, we obtain 
$$
\left(y^{{j+1}},\Delta_{0t_{j+1}}^\alpha y\right)\geq \frac{\tau^{-\alpha}}{\Gamma(2-\alpha)}\left(E_{j+1}-E_{j}\right)+\frac{1}{2}\bar{\Delta}_{0t_{j+1}}^{\alpha}\|y\|_0^2
$$
$$
=\frac{\tau^{-\alpha}}{\Gamma(2-\alpha)}\left(\mathcal{E}_{j+1}-\mathcal{E}_{j}\right) - \frac{\tau^{-\alpha}}{2\Gamma(2-\alpha)}\bar c^{(\alpha)}_j \|y^0\|_0^2, ,\quad j=1, 2, \ldots, M-1,
$$
where
$$
E_{j}=\left(\frac{1}{2}\sqrt{\frac{c_0^{(\alpha)}-c_1^{(\alpha)}}{2}}+\frac{1}{2}\sqrt{\frac{c_0^{(\alpha)}+3c_1^{(\alpha)}-4c_2^{(\alpha)}}{2}}\right)^2\|y^{j}\|_0^{2}
$$
$$
+\left\|\sqrt{\frac{c_0^{(\alpha)}-c_1^{(\alpha)}}{2}}y^{j}-\left(\frac{1}{2}\sqrt{\frac{c_0^{(\alpha)}-c_1^{(\alpha)}}{2}}+\frac{1}{2}\sqrt{\frac{c_0^{(\alpha)}+3c_1^{(\alpha)}-4c_2^{(\alpha)}}{2}}\right)y^{j-1}\right\|_0^2.
$$ 
$$
\mathcal{E}_{j} = E_{j} + \frac{1}{2}\sum\limits_{s=0}^{j-1}\bar{c}_{j-1-s}^{(\alpha)}\|y^{s+1}\|_0^2.
$$

For the difference operator $\Lambda$ using
Green's first difference formula for the functions vanishing at  $x=0$ and $x=l$, we
get   $(-\Lambda y,y)\geq c_1\|y_{\bar x}]|_0^2$.

From (\ref{ur18_1}), using that 
$$
\left(y^{{j+1}},\varphi^{j+1}\right)\leq\frac{c_1}{l^2} \|y^{j+1}\|_0^2+\frac{l^2}{4c_1}\|\varphi^{j+1}\|_0^2\leq\frac{c_1}{2} \|y_{\bar x}^{j+1}]|_0^2+\frac{l^2}{4c_1}\|\varphi^{j+1}\|_0^2, 
$$ one obtains the inequality 
\begin{equation}\label{ur18_2}
\frac{\tau^{-\alpha}}{\Gamma(2-\alpha)}\left(\mathcal{E}_{j+1}-\mathcal{E}_{j}\right) + \frac{c_1}{2}\|y_{\bar x}^{j+1}]|_0^2\leq \frac{l^2}{4c_1}\|\varphi^{j+1}\|_0^2 + \frac{\tau^{-\alpha}}{2\Gamma(2-\alpha)}\bar c^{(\alpha)}_j \|y^0\|_0^2.
\end{equation}
Multiplying inequality (\ref{ur18_2}) by $\tau$ and summing the resulting relation over $j$ from $1$ to $M-1$ and taking into account inequality (\ref{url2_4}), one obtains a priori estimate (\ref{ur20}).

The stability and convergence
of the difference scheme (\ref{ur18}) - (\ref{ur19}) follow from the a priori estimate (\ref{ur20}).
\end{proof}

\subsection{Numerical results}

Numerical computations are executed for a test problem on the assumption that the
function
$$u(x,t)=\sin(\pi x)\left(t^{3+\alpha}+t^2+1\right)$$
is the exact solution of problem (\ref{ur01})--(\ref{ur02}) with
the coefficients $k(x,t)=2-\cos(xt)$, $q(x,t)=1-\sin(xt)$ and $l=1$,
$T=1$.

The errors ($z=y-u$) and convergence order (CO) in the norms
$\|\cdot\|_0$ and $\|\cdot\|_{\mathcal{C}(\bar\omega_{h\tau})}$,
where
$\|y\|_{\mathcal{C}(\bar\omega_{h\tau})}=\max\limits_{(x_i,t_j)\in\bar\omega_{h\tau}}|y|$,
are given in Table 1.

\textbf{Table 1} demonstrates that as the number of the spatial
subintervals and time steps increases, while $h^2=\tau^{3-\alpha}$,
then the maximum error decreases, as it is expected and the
convergence order of the approximate scheme is
$\mathcal{O}(h^2)=\mathcal{O}(\tau^{3-\alpha})$, where the convergence order
is given by the formula:
CO$=\log_{\frac{h_1}{h_2}}{\frac{\|z_1\|}{\|z_2\|}}$ ($z_{i}$ is the
error corresponding to $h_{i}$).

\textbf{Table 2} shows that if $h=1/50000$, then as the number of
time steps of our approximate scheme increases, then
the maximum error decreases, as it is expected and the convergence order
of time is $\mathcal{O}(\tau^{3-\alpha})$, where the convergence order is
given by the following formula:
CO$=\log_{\frac{\tau_1}{\tau_2}}{\frac{\|z_1\|}{\|z_2\|}}$.

\begin{table}[h!]
    \begin{center}
        \caption{The error and the convergence order in the norms $\|\cdot\|_{0}$
            and $\|\cdot\|_{C(\bar{\omega}_{h\tau})}$ when
            decreasing time-grid size for different values of $\alpha=0.1; 0.5; 0.9$, $\tau^{3-\alpha}=h^2.$}
        \label{tab:table1}
\begin{tabular}{|c|c|c|c|c|c|c|c|c|} 
    \hline
    \, \textbf{$\alpha$}\, & \,\textbf{$\tau$}\, & \, \textbf{$h$}  \, & \, \textbf{$\max\limits_{0\leq j\leq M}\|z^j\|_{0}$}  \, & \, \textbf{ CO } \, &\,  \textbf{$\max\limits_{0\leq j\leq M}\|z^j\|_{{C(\bar{\omega}_{h\tau})}}$}\, & \,\textbf{ CO} \, &\, \textbf{$\max\limits_{0\leq j\leq M}\|z^j_{\bar x}]|_{0}$} \,& \, \textbf{ CO } \,\\
    \hline
        0.1 & 1/10  & 1/29   & 1.694597e-3 &          & 2.387728e-3   &           & 5.341255e-2   &          \\
            & 1/20  & 1/78   & 2.343539e-4 & 2.8541   & 3.304157e-4   &  2.8533   & 7.389758e-4   &  2.8536  \\
            & 1/40  & 1/211  & 3.204204e-5 & 2.8707   & 4.517782e-5   &  2.8706   & 1.010417e-4   &  2.8706   \\
            & 1/80  & 1/575  & 4.316975e-6 & 2.8918   & 6.086836e-6   &  2.8919   & 1.361321e-5   &  2.8919   \\
            & 1/160 & 1/1571 & 5.786215e-7 & 2.8993   & 8.158422e-7   &  2.8993   & 1.824621e-6   &  2.8993   \\
           \hline
        0.5 & 1/10 & 1/18    & 4.556026e-3 &          & 6.401088e-3   &           & 1.434106e-2   &          \\
            & 1/20 & 1/43    & 8.011052e-4 & 2.5077   & 1.129064e-3   &  2.5032   & 2.524196e-3   &  2.5063  \\
            & 1/40 & 1/101   & 1.452643e-4 & 2.4633   & 2.047995e-4   &  2.4628   & 4.577935e-4   &  2.4630  \\
            & 1/80 & 1/240   & 2.575571e-5 & 2.4957   & 3.631166e-5   &  2.4957   & 8.116952e-5   &  2.4956  \\
            & 1/160 & 1/570  & 4.568945e-6 & 2.4950   & 6.441587e-6   &  2.4949   & 1.439907e-5   &  2.4950   \\
           \hline
        0.9 & 1/10 & 1/12    & 1.181474e-2 &          & 1.662948e-2   &           & 3.707516e-2   &          \\
            & 1/20 & 1/24    & 2.931153e-3 & 2.0110   & 4.125467e-3   &  2.0111   & 9.218339e-3   &  2.0078  \\
            & 1/40 & 1/49    & 7.018065e-4 & 2.0623   & 9.891705e-4   &  2.0603   & 2.208378e-3   &  2.0615  \\
            & 1/80 & 1/100   & 1.678681e-4 & 2.0637   & 2.367034e-4   &  2.0631   & 5.283157e-4   &  2.0635  \\
            & 1/160 & 1/207  & 3.921292e-5 & 2.0979   & 5.529153e-5   &  2.0979   & 1.234141e-4   &  2.0979   \\
           \hline
        \end{tabular}
    \end{center}
\end{table}

\begin{table}[h!]
    \begin{center}
        \caption{The error and the convergence order in the norms $\|\cdot\|_{0}$
            and $\|\cdot\|_{C(\bar{\omega}_{h\tau})}$ when
            decreasing time-grid size for different values of $\alpha=0.3; 0.5; 0.7$, $h=1/50000.$}
        \label{tab:table1}
\begin{tabular}{|c|c|c|c|c|c|c|c|} 
    \hline
    \, \textbf{$\alpha$}\, & \,\textbf{$\tau$}\,  & \, \textbf{$\max\limits_{0\leq j\leq M}\|z^j\|_{0}$}  \, & \, \textbf{ CO } \, &\,  \textbf{$\max\limits_{0\leq j\leq M}\|z^j\|_{{C(\bar{\omega}_{h\tau})}}$}\, & \,\textbf{ CO} \, &\, \textbf{$\max\limits_{0\leq j\leq M}\|z^j_{\bar x}]|_{0}$} \,& \, \textbf{ CO } \,\\
    \hline
        0.3 & 1/10     & 7.281556e-5 &          & 1.036431e-4   &           & 2.293180e-4   &          \\
            & 1/20     & 1.202886e-5 & 2.5977   & 1.712493e-5   &  2.5974   & 3.787942e-5   &  2.5978  \\
            & 1/40     & 1.881330e-6 & 2.6766   & 2.674734e-6   &  2.6786   & 5.928309e-6   &  2.6757   \\
            & 1/80    & 2.908398e-7 & 2.6934    & 4.140875e-7   &  2.6914   & 9.159351e-7   &  2.6943   \\
           \hline
        0.5 & 1/10     & 2.726395e-4 &          & 3.880588e-4   &           & 8.586014e-4   &          \\
            & 1/20     & 5.051848e-5 & 2.4321   & 7.190513e-5   &  2.4321   & 1.590939e-4   &  2.4321  \\
            & 1/40     & 9.152847e-6 & 2.4645   & 1.302443e-5   &  2.4648   & 2.882759e-5   &  2.4643  \\
            & 1/80     & 1.623335e-6 & 2.4952   & 2.310709e-6   &  2.4948   & 5.112271e-6   &  2.4954  \\
           \hline
        0.7 & 1/10     & 8.556143e-4 &          & 1.217803e-3   &           & 2.694425e-3   &          \\
            & 1/20     & 1.810137e-4 & 2.2408   & 2.576392e-4   &  2.2408   & 5.700338e-4   &  2.2408  \\
            & 1/40     & 3.759528e-5 & 2.2674   & 5.351332e-5   &  2.2673   & 1.183890e-4   &  2.2675  \\
            & 1/80     & 7.685019e-6 & 2.2904   & 1.093830e-5   &  2.2905   & 2.420107e-5   &  2.2903  \\
           \hline
        \end{tabular}
    \end{center}
\end{table}

\section{A compact difference scheme for the time
	fractional diffusion equation}

In this section for problem (\ref{ur01})--(\ref{ur02}), we create
a compact difference scheme with the approximation order
$\mathcal{O}(h^4+\tau^{3-\alpha})$ in the case when  $k=k(t)$ and $q=q(t)$.
The stability and convergence of the constructed difference scheme
in the grid  $L_2$ - norm with the rate equal to the order of the
approximation error are proved. The found results are
supported by the numerical calculations implemented for a test
example.

\subsection{Derivation  of the difference scheme}

Let a difference scheme be put into a correspondence with differential problem
(\ref{ur01})--(\ref{ur02}) in the case when  $k=k(t)$ and $q=q(t)$:
\begin{equation}\label{ur21}
	\Delta_{0t_{j+1}}^{\alpha}\mathcal{H}_hy_i=a^{j+1}y_{\bar
		xx,i}^{j+1}
	-d^{j+1}\mathcal{H}_hy_i^{j+1}+\mathcal{H}_h\varphi_i^{j+1}, \,
	i=1,\ldots,N-1,\, j=0,1,\ldots,M-1,
\end{equation}
\begin{equation}
	y(0,t)=0,\quad y(l,t)=0,\quad t\in \overline \omega_{\tau}, \quad
	y(x,0)=u_0(x),\quad  x\in \overline \omega_{h},\label{ur22}
\end{equation}
where $\mathcal{H}_hv_i=v_i+h^2v_{\bar xx,i}/12$, $i=1,\ldots,N-1$,
$a^{j+1}=k(t_{j+1})$, $d^{j+1}=q(t_{j+1})$,
$\varphi_i^{j+1}=f(x_i,t_{j+1})$.

From Lemma \ref{lem1}  it follows that if $u\in
\mathcal{C}_{x,t}^{6,3}$, then the difference scheme has the
approximation order  $\mathcal{O}(\tau^{3-\alpha}+h^4)$.

\subsection{Stability and convergence}

\begin{theorem} The difference scheme (\ref{ur21})--(\ref{ur22})
is unconditionally stable and its solution meets the following a
priori estimate:
\begin{equation}\label{ur23}
		\sum\limits_{j=1}^{M-1}\left(\|\mathcal{H}_hy^{j+1}\|_0^2 + \|y_{\bar x}^{j+1}]|_0^2\right)\tau \leq M_2\left(\|\mathcal{H}_h y^1\|_0^2+\|\mathcal{H}_h y^0\|_0^2+\sum\limits_{j=1}^{M-1}\|\mathcal{H}_h\varphi^{j+1}\|_0^2\tau\right),
\end{equation}
where $M_2>0$ is a known number independent of $h$ and $\tau$.
\end{theorem}

\begin{proof} Taking the inner product of the equation
(\ref{ur21}) with
$\mathcal{H}_hy^{j+1}=(\mathcal{H}_hy)^{j+1}$, we have
$$
(\mathcal{H}_hy^{j+1},\Delta_{0t_{j+1}}^{\alpha}\mathcal{H}_hy)-a^{j+1}(\mathcal{H}_hy^{j+1},y_{\bar
	xx}^{j+1})
$$
\begin{equation}\label{ur24}
	+d^{j+1}(\mathcal{H}_hy^{j+1},\mathcal{H}_hy^{j+1})=(\mathcal{H}_hy^{j+1},\mathcal{H}_h\varphi^{j+1}).
\end{equation}
We transform the terms in identity (\ref{ur24}) as
$$
(\mathcal{H}_hy^{j+1},\Delta_{0t_{j+1}}^{\alpha}\mathcal{H}_hy)\geq
\frac{\tau^{-\alpha}}{\Gamma(2-\alpha)}\left(E_{j+1}-E_{j}\right)+
\frac{1}{2}\bar\Delta_{0t_{j+\sigma}}^{\alpha}\|\mathcal{H}_hy\|_0^2=
$$
$$
=\frac{\tau^{-\alpha}}{\Gamma(2-\alpha)}\left(\mathcal{E}_{j+1}-\mathcal{E}_{j}\right) - \frac{\tau^{-\alpha}}{2\Gamma(2-\alpha)}\bar c^{(\alpha)}_j \|\mathcal{H}_hy^0\|_0^2, ,\quad j=1, 2, \ldots, M-1,
$$
where
$$
E_{j}=\left(\frac{1}{2}\sqrt{\frac{c_0^{(\alpha)}-c_1^{(\alpha)}}{2}}+\frac{1}{2}\sqrt{\frac{c_0^{(\alpha)}+3c_1^{(\alpha)}-4c_2^{(\alpha)}}{2}}\right)^2\|\mathcal{H}_h y^{j}\|_0^{2}
$$
$$
+\left\|\sqrt{\frac{c_0^{(\alpha)}-c_1^{(\alpha)}}{2}}\mathcal{H}_hy^{j}-\left(\frac{1}{2}\sqrt{\frac{c_0^{(\alpha)}-c_1^{(\alpha)}}{2}}+\frac{1}{2}\sqrt{\frac{c_0^{(\alpha)}+3c_1^{(\alpha)}-4c_2^{(\alpha)}}{2}}\right)\mathcal{H}_hy^{j-1}\right\|_0^2,
$$ 
$$
\mathcal{E}_{j} = E_{j} + \frac{1}{2}\sum\limits_{s=0}^{j-1}\bar{c}_{j-1-s}^{(\alpha)}\|\mathcal{H}_h y^{s+1}\|_0^2.
$$
$$
-(\mathcal{H}_hy^{j+1},y_{\bar
	xx}^{j+1})=-(y^{j+1},y_{\bar
	xx}^{j+1})-\frac{h^2}{12}\|y_{\bar
	xx}^{j+1}\|_0^2=\|y_{\bar
	x}^{j+1}]|_0^2-\frac{1}{12}\sum\limits_{i=1}^{N-1}(y_{\bar
	x,i+1}^{j+1}-y_{\bar x,i}^{j+1})^2h
$$
$$
\geq\|y_{\bar x}^{j+1}]|_0^2-\frac{1}{3}\|y_{\bar
	x}^{j+1}]|_0^2=\frac{2}{3}\|y_{\bar
	x}^{j+1}]|_0^2,
$$
$$
(\mathcal{H}_hy^{j+1},\mathcal{H}_h\varphi^{j+1})\leq\varepsilon\|\mathcal{H}_hy^{j+1}\|_0^2+
\frac{1}{4\varepsilon}\|\mathcal{H}_h\varphi^{j+1}\|_0^2
$$
$$
=\varepsilon\sum\limits_{i=1}^{N-1}\left(\frac{y_{i-1}^{j+1}+10y_{i}^{j+1}+y_{i+1}^{j+1}}{12}\right)^2h+
\frac{1}{4\varepsilon}\|\mathcal{H}_h\varphi^{j+1}\|_0^2
$$
$$
\leq\varepsilon\|y^{j+1}\|_0^2+\frac{1}{4\varepsilon}\|\mathcal{H}_h\varphi^{j+1}\|_0^2\leq \frac{\varepsilon l^2}{2}\|y_{\bar x}^{j+1}]|_0^2+\frac{1}{4\varepsilon}\|\mathcal{H}_h\varphi^{j+1}\|_0^2.
$$
In view of the above-performed transformations, from
identity (\ref{ur24}) at $\varepsilon=\frac{c_1}{3l^2}$ we
get the inequality
$$
\frac{\tau^{-\alpha}}{\Gamma(2-\alpha)}\left(\mathcal{E}_{j+1}-\mathcal{E}_{j}\right) + \frac{c_1}{2}\|y_{\bar x}^{j+1}]|_0^2\leq\frac{3l^2}{4c_1}\|\mathcal{H}_h\varphi^{j+1}\|_0^2+\frac{\tau^{-\alpha}}{2\Gamma(2-\alpha)}\bar c^{(\alpha)}_j \|\mathcal{H}_hy^0\|_0^2.
$$
The following process is similar to the proof of Theorem \ref{thm_1}, and we
leave it out.
\end{proof}

The norm $\|\mathcal{H}_hy\|_0$ is equivalent to the norm $\|y\|_0$, which follows from the inequalities
$$
\frac{5}{12}\|y\|_0^2\leq\|\mathcal{H}_hy\|_0^2\leq\|y\|_0^2.
$$

Using a
priori estimate (\ref{ur23}), we obtain the convergence result.

\begin{theorem}
 Let $u(x,t)\in\mathcal{C}_{x,t}^{6,3}$
be the solution of problem (\ref{ur01})--(\ref{ur02}) in the
case $k=k(t)$, $q=q(t)$, and let $\{y_i^j \,|\, 0\leq i\leq N, \,
1\leq j\leq M\}$ be the solution of difference scheme
(\ref{ur21})--(\ref{ur22}). Then it holds true that
$$
	\sqrt{\sum\limits_{j=1}^{M-1}\left(\|z^{j+1}\|_0^2 + \|z_{\bar x}^{j+1}]|_0^2\right)\tau}\leq C_R\left(\tau^{3-\alpha}+h^4\right),\quad 1\leq
j\leq M,
$$
where $z_i^j = u(x_i,t_j)-y_i^j$ and $C_R$ is a positive constant independent of $\tau$ and $h$.
\end{theorem}

\subsection{Numerical results}

Numerical calculations are performed for a test problem when the
function
$$u(x,t)=\sin(\pi x)\left(t^{3+\alpha}+t^2+1\right)$$
is the exact solution of the problem (\ref{ur01})--(\ref{ur02}) with
the coefficients $k(x,t)=2-\cos(t)$, $q(x,t)=1-\sin(t)$ and $l=1$,
$T=1$.

The errors ($z=y-u$) and convergence order (CO) in the norms
$\|\cdot\|_0$ and $\|\cdot\|_{\mathcal{C}(\bar\omega_{h\tau})}$,
where
$\|y\|_{\mathcal{C}(\bar\omega_{h\tau})}=\max\limits_{(x_i,t_j)\in\bar\omega_{h\tau}}|y|$,
are given in Table 1.

\textbf{Table 3} shows that as the number of the spatial
subintervals and time steps increases keeping $h^4=\tau^{3-\alpha}$, 
the maximum error decreases, as it is expected and the
convergence order of the compact difference scheme is
$\mathcal{O}(h^2)=\mathcal{O}(\tau^{3-\alpha})$, where the convergence order
is given by the formula:
CO$=\log_{\frac{h_1}{h_2}}{\frac{\|z_1\|}{\|z_2\|}}$ ($z_{i}$ is the
error corresponding to $h_{i}$).

\textbf{Table 4} demonstrates that if $h=1/2000$, then as the number of
time steps of our approximate scheme increases, then
the maximum error decreases, as it is expected and the convergence order
of time is $\mathcal{O}(\tau^{3-\alpha})$, where the convergence order is
given by the following formula:
CO$=\log_{\frac{\tau_1}{\tau_2}}{\frac{\|z_1\|}{\|z_2\|}}$.

\begin{table}[h!]
    \begin{center}
        \caption{The error and the convergence order in the norms $\|\cdot\|_{0}$
            and $\|\cdot\|_{C(\bar{\omega}_{h\tau})}$ when
            decreasing time-grid size for different values of $\alpha=0.1; 0.5; 0.9$, $\tau^{3-\alpha}=(h/2)^4.$}
        \label{tab:table1}
\begin{tabular}{|c|c|c|c|c|c|c|c|c|} 
    \hline
    \, \textbf{$\alpha$}\, & \,\textbf{$\tau$}\, & \, \textbf{$h$}  \, & \, \textbf{$\max\limits_{0\leq j\leq M}\|z^j\|_{0}$}  \, & \, \textbf{ CO } \, &\,  \textbf{$\max\limits_{0\leq j\leq M}\|z^j\|_{{C(\bar{\omega}_{h\tau})}}$}\, & \,\textbf{ CO} \, &\, \textbf{$\max\limits_{0\leq j\leq M}\|z^j_{\bar x}]|_{0}$} \,& \, \textbf{ CO } \,\\
    \hline
        0.1 & 1/40  & 1/29   & 1.321499e-6  &          & 1.866140e-6   &           & 4.149581e-6   &          \\
            & 1/80  & 1/47   & 1.912169e-7  & 2.7889   & 2.702706e-7   &  2.7875   & 6.006140e-7   &  2.7884  \\
            & 1/160 & 1/79   & 2.443382e-8  & 2.9683   & 3.454781e-8   &  2.9677   & 7.675607e-8   &  2.9680  \\
            & 1/320 & 1/131  & 3.267337e-9  & 2.9027   & 4.620380e-9   &  2.9025   & 1.026439e-8   &  2.9026  \\
            & 1/640 & 1/217  & 4.395804e-10 & 2.8939   & 6.216471e-10  &  2.8938   & 1.380970e-9   &  2.8939  \\
           \hline
        0.5 & 1/40  & 1/21   & 1.178052e-5 &          & 1.661359e-5   &           & 3.697512e-5   &          \\
            & 1/80  & 1/31   & 2.241843e-6 & 2.3936   & 3.166375e-6   &  2.3914   & 7.039944e-6   &  2.3929  \\
            & 1/160 & 1/47   & 4.096169e-7 & 2.4523   & 5.789623e-7   &  2.4512   & 1.286610e-6   &  2.4519  \\
            & 1/320 & 1/73   & 7.195323e-8 & 2.5091   & 1.017336e-7   &  2.5086   & 2.260303e-7   &  2.5089  \\
            & 1/640 & 1/113  & 1.268803e-8 & 2.5036   & 1.794185e-8   &  2.5034   & 3.985934e-8   &  2.5035  \\
           \hline
        0.9 & 1/40 & 1/13    & 1.470087e-4 &          & 2.063859e-4   &           & 4.607185e-4   &          \\
            & 1/80 & 1/19    & 3.419750e-5 & 2.1039   & 4.819739e-5   &  2.0983   & 1.073122e-4   &  2.1020  \\
            & 1/160 & 1/29   & 7.726281e-6 & 2.1460   & 1.091058e-5   &  2.1432   & 2.426096e-5   &  2.1451  \\
            & 1/320 & 1/41   & 1.824394e-6 & 2.0823   & 2.578190e-6   &  2.0812   & 5.730103e-6   &  2.0820  \\
            & 1/640 & 1/59   & 4.260141e-7 & 2.0984   & 6.022614e-7   &  2.0978   & 1.338204e-6   &  2.0982  \\
           \hline
        \end{tabular}
    \end{center}
\end{table}

\begin{table}[h!]
    \begin{center}
        \caption{The error and the convergence order in the norms $\|\cdot\|_{0}$
            and $\|\cdot\|_{C(\bar{\omega}_{h\tau})}$ when
            decreasing time-grid size for different values of $\alpha=0.3; 0.5; 0.7$, $h=1/1000.$}
        \label{tab:table1}
\begin{tabular}{|c|c|c|c|c|c|c|c|} 
    \hline
    \, \textbf{$\alpha$}\, & \,\textbf{$\tau$}\,  & \, \textbf{$\max\limits_{0\leq j\leq M}\|z^j\|_{0}$}  \, & \, \textbf{ CO } \, &\,  \textbf{$\max\limits_{0\leq j\leq M}\|z^j\|_{{C(\bar{\omega}_{h\tau})}}$}\, & \,\textbf{ CO} \, &\, \textbf{$\max\limits_{0\leq j\leq M}\|z^j_{\bar x}]|_{0}$} \,& \, \textbf{ CO } \,\\
    \hline
        0.3 & 1/10     & 6.155178e-5 &          & 8.704736e-5   &           & 1.933705e-4   &          \\
            & 1/20     & 1.016170e-5 & 2.5986   & 1.437081e-5   &  2.5986   & 3.192392e-5   &  2.5986  \\
            & 1/40     & 1.642526e-6 & 2.6291   & 2.322883e-6   &  2.6291   & 5.160147e-6   &  2.6291   \\
            & 1/80     & 2.620773e-7 & 2.6478   & 3.706331e-7   &  2.6478   & 8.233399e-7   &  2.6478   \\
            & 1/160    & 4.147475e-8 & 2.6596   & 5.865421e-8   &  2.6596   & 1.302967e-7   &  2.6596   \\
           \hline
        0.5 & 1/10     & 2.308509e-4 &          & 3.264725e-4   &           & 7.252393e-4   &          \\
            & 1/20     & 4.277465e-5 & 2.4321   & 6.049249e-5   &  2.4321   & 1.343804e-4   &  2.4321  \\
            & 1/40     & 7.775493e-6 & 2.4597   & 1.099620e-5   &  2.4597   & 2.442742e-5   &  2.4597  \\
            & 1/80     & 1.398769e-6 & 2.4747   & 1.978159e-6   &  2.4747   & 4.394362e-6   &  2.4747  \\
            & 1/160    & 2.500909e-6 & 2.4836   & 3.536819e-7   &  2.4836   & 7.856834e-7   &  2.4836  \\
           \hline
        0.7 & 1/10     & 7.275485e-4 &          & 1.028909e-3   &           & 2.285660e-3   &          \\
            & 1/20     & 1.539001e-4 & 2.2410   & 2.176477e-4   &  2.2410   & 4.834914e-4   &  2.2410  \\
            & 1/40     & 3.192728e-5 & 2.2691   & 4.515200e-5   &  2.2691   & 1.003024e-4   &  2.2691  \\
            & 1/80     & 6.558744e-6 & 2.2832   & 9.275465e-6   &  2.2832   & 2.060489e-5   &  2.2832  \\
            & 1/160    & 1.340380e-6 & 2.2907   & 1.895584e-6   &  2.2907   & 4.210928e-6   &  2.2907  \\
           \hline
        \end{tabular}
    \end{center}
\end{table}

\section{Conclusion}

In the current paper we construct a $L2$ type difference approximation of the Caputo fractional derivative
with the approximation order $\mathcal{O}(\tau^{3-\alpha})$.
The fundamental features of this difference operator are
studied. New difference schemes of the second and fourth
approximation order in space and the $3-\alpha$ approximation order in
time for the time fractional diffusion equation with variable
coefficients are also constructed. The stability and convergence
of these schemes with the rate equal to the
order of the approximation error are proved. The method can be
without difficulty expanded to include other time fractional partial differential
equations with other boundary conditions.

Numerical tests entirely
corroborating the found theoretical results are implemented. In all
the calculations Julia v1.5.1 is used.

\end{document}